\documentclass[notitlepage,11pt,reqno]{amsart}
\usepackage{amssymb,amsmath,amsthm,color,enumerate}
\usepackage[breaklinks=true]{hyperref}
\usepackage{mathrsfs}

\usepackage[letterpaper]{geometry}
\newcommand{\bb}{\ensuremath{\mathbb B}}
\newcommand{\G}{\ensuremath{\mathscr G}}
\newcommand{\bp}[1]{\ensuremath{\mathbb P}_\mu \left( #1 \right)}
\newcommand{\conv}[1]{\ensuremath{\stackrel{#1}{\longrightarrow}} }
\renewcommand{\leq}{\leqslant}
\renewcommand{\geq}{\geqslant}
\renewcommand{\subset}{\subseteq}

\newtheorem{theorem}{Theorem}[section]

\newtheorem{prop}[theorem]{Proposition}

\theoremstyle{definition}
\newtheorem{definition}[theorem]{Definition}
\newtheorem{remark}[theorem]{Remark}

\numberwithin{equation}{section}

\allowdisplaybreaks

\begin{document}

\title[Probability inequalities for strongly left-invariant metric
semigroups/monoids]{Probability inequalities for strongly left-invariant
metric semigroups/monoids, including all Lie groups}

\author{Apoorva Khare}
\address[A.~Khare]{Indian Institute of Science, Bangalore -- 560012,
India; and Analysis and Probability Research Group, Bangalore -- 560012,
India}
\email{\tt khare@iisc.ac.in}

\dedicatory{To the memory of K.R.\ Parthasarathy, with admiration}

\date{\today}

\begin{abstract}
Recently, a general version of the Hoffmann-J{\o}rgensen inequality was
shown jointly with Rajaratnam [\textit{Ann.\ Probab.}\ 2017], which
(a)~improved the result even for real-valued variables, but also
(b)~simultaneously unified and extended several versions in the Banach
space literature, including that by
Hitczenko--Montgomery-Smith [\textit{Ann.\ Probab.}\ 2001],
as well as special cases and variants of results by
Johnson--Schechtman [\textit{Ann.\ Probab.}\ 1989] and
Klass--Nowicki [\textit{Ann.\ Probab.}\ 2000],
in addition to the original versions by Kahane and Hoffmann-J{\o}rgensen.
Moreover, our result with Rajaratnam was in a primitive framework: over
all semigroups with a bi-invariant metric; this includes Banach spaces as
well as compact and abelian Lie groups.

In this note we show the result even more generally: over every semigroup
$\mathscr{G}$ with a strongly left- (or right-)invariant metric. We also
prove some applications of this inequality over such $\mathscr{G}$,
extending Banach space-valued versions by
Hitczenko and Montgomery-Smith [\textit{Ann.\ Probab.}\ 2001] and by
Hoffmann-J{\o}rgensen [\textit{Studia Math.}\ 1974].
Furthermore, we show several other stochastic inequalities -- by 
Ottaviani--Skorohod, Mogul'skii, and L\'evy--Ottaviani --
as well as L\'evy's equivalence, again over $\mathscr{G}$ as above.
This setting of generality for $\mathscr{G}$ subsumes not only semigroups
with bi-invariant metric (thus extending the previously shown results),
but it also means that these results now hold over all Lie groups
(equipped with a left-invariant Riemannian metric).

We also explain why this primitive setting of strongly
left/right-invariant metric semigroups $\mathscr{G}$ is equivalent to
that of left/right-invariant metric monoids $\mathscr{G}_\circ$: each
such $\mathscr{G}$ embeds in some $\mathscr{G}_\circ$.
\end{abstract}



\maketitle

\section{Introduction: strongly left/right-invariant metric semigroups}

In this work our goal is to extend various results in the probability
literature to very primitive settings -- e.g.\ from real- or Banach
space-valued random variables, to ones taking values in an arbitrary Lie
group, and even more general classes of variables. This continues a
series of recent joint works~\cite{KR1,KR2,KR3} by the author, which were
inspired in part by the seminal treatise~\cite{P} of Parthasarathy that
studied (probability) measures on very primitive structures: (separable)
metric spaces, metric/LCA groups, and so on.

In this note, our goal is to extend stochastic inequalities, tail
estimates, and convergence phenomena that were known to hold for random
variables over Banach spaces, or -- very recently -- semigroups with
bi-invariant metrics, to semigroups with strongly left-invariant (or
right-invariant) metrics. We begin by defining these latter notions.

\begin{definition}\hfill
\begin{enumerate}
\item As defined in \cite{KR1}, a \textit{bi-invariant metric semigroup}
consists of a semigroup $(\G, \cdot)$ with a bi-invariant metric $d_\G$
-- i.e.,
\begin{equation}\label{Ebiinv}
d_\G(ca,cb) = d_\G(a,b) = d_\G(ac,bc), \qquad \forall a,b,c \in \G.
\end{equation}
(In our previous joint works~\cite{KR1,KR2,KR3}, we refer to such a $\G$
as merely a metric semigroup; however, the bi-invariance will be
explicitly pointed out in this work, to distinguish from the following
notions.)

\item If only the first (respectively, second) equality in~\eqref{Ebiinv}
holds for all $a,b,c \in \G$, then we say that $\G$ is a \textit{left-}
(respectively, \textit{right-})\textit{invariant metric semigroup}.

\item Similarly, one defines a \textit{left/right/bi-invariant metric
monoid/group}.

\item Finally, a left-invariant metric semigroup $\G$ is \textit{strongly
left-invariant} if $d_\G(a, ab) = d_\G(b, b^2)$ for all $a,b \in \G$. One
similarly defines a strongly right-invariant metric semigroup.
\end{enumerate}
\end{definition}

\begin{remark}\label{R1}
In other words, $d_\G$ is strongly left-invariant if and only if for all
$b \in \G$, $d_\G(a,ab)$ is independent of $a \in \G$ (so one can set
$a=b$ in $\G$).
Thus every left-invariant metric monoid or group $(\G, \cdot, e, d_\G)$
is a strongly left-invariant semigroup, since $d_\G(a,ab) = d_\G(e,b)$.
(Ditto for right-invariant $d_\G$.) Similarly, it was shown in \cite{KR1}
(see~\eqref{Einv} below) that every bi-invariant metric semigroup is
strongly left- and right-invariant. Thus, adding an identity or
right-invariance automatically upgrades left-invariance to its ``strong''
version, and so the present setting subsumes both of these related
settings above/in previous works.
\end{remark}

The goal of this work is to extend results from random variables taking
values in Banach spaces (as is traditional by now) or even in
bi-invariant metric semigroups (as was done in recent joint works), to
strongly left/right-invariant metric semigroup-valued variables.
The motivation to extend results from Banach spaces to more primitive
frameworks is both classical and modern. Following its axiomatization and
systematic development, one of the cornerstones of twentieth century
probability theory has been to extend results for real-valued random
variables to $\mathbb{B}$-valued random variables, for $\mathbb{B}$ a
(separable) Banach space -- see e.g.\ the classic treatise \cite{LT}. Now
a natural theoretical question is to explore settings beyond Banach
spaces. In fact such questions have been widely studied in the past few
decades. We mention the classic monographs by Parthasarathy~\cite{P} and
Grenander~\cite{Gre}; as well as (among many others) the
Diaconis--Shahshahani work on random permutations~\cite{DS} and the
recent theory of (dense as well as sparse) graph limits -- which
has already been crystallized by Lov\'asz in book form
\cite{Lo}.\footnote{While every metric space isometrically embeds
into a Banach space by the Kuratowski embedding theorem, the study of
graphons with the cut-norm does not typically proceed using this
embedding.}
This activity continues to thrive; e.g.\ outside the (by now) traditional
Banach space setting, we list a few of the many works involving random
variables taking values in (possibly non-compact and non-abelian) Lie
groups, in random matrix theory \cite{BQ,GKZ,Jo} and its connections to
ergodic theory and geometry (see e.g.\ \cite{FK,HS,Pollicott}).

Thus our goal in this note is to extend several results in the vast
literature on Banach space-valued random variables -- or more generally,
results with variables valued in bi-invariant metric semigroups (these
further include discrete, abelian, compact, or amenable groups, see the
introduction to \cite{KR1} for more examples and \cite{KR1,KR2,KR3} for
the results) -- to an even more primitive setting: strongly
left/right-invariant metric semigroups. In particular, since these
include left-invariant metric groups, it follows that several results
which were previously known only over compact or abelian Lie groups, now
extend to every Lie group (including non-compact, non-abelian ones).

\begin{remark}
This work attempts to provide the \textit{most primitive setting} in
which the results stated below can be proved. To this end:
\begin{enumerate}
\item By Remark~\ref{R1}, random variables taking values in strongly
left/right-invariant metric semigroups $\G_{strong}$ subsume those taking
values in either left/right-invariant metric monoids $\G_\circ$ or in
bi-invariant metric semigroups $\G$. So working with such semigroups
$\G_{strong}$ is at least as general.

\item While the results below can be stated over any semigroup with a
metric, it is not clear (to the author) how to prove them in a more
primitive setting than strong left/right-invariance, since it is
indispensable in the proofs of most/all of the results below. One way to
avoid this technical hurdle (the word ``strong'') could be if -- as in
the bi-invariant case -- that every left/right-invariant metric semigroup
embeds isometrically and homomorphically inside a left/right-invariant
metric monoid $\G_\circ$, because then
$d_\G(a,ab) = d_{\G_\circ}(e,b) = d_\G(b,b^2)$. As we show in
Proposition~\ref{Pcounterex}, this does not always happen.
\end{enumerate}
Thus, we suspect that the strong left- or right-invariance of the metric
is perhaps the \textit{minimum} amount of structure required in order to
be able to show the Hoffmann-J{\o}rgensen, Ottaviani--Skorohod,
Mogul'skii, and L\'evy--Ottaviani inequalities and their applications
shown in this work.
\end{remark}

We end here with the punchline of the ``non-probabilistic part'' of the
paper: \textit{every strongly left/right-invariant metric semigroup
$\G_{strong}$ embeds isometrically and homomorphically into a
left/right-invariant metric monoid $(\G_\circ, e)$.} (In fact, we show
that $\G_{strong} \supseteq \G_\circ \setminus \{ e \}$, so that the
smallest $\G_\circ \supseteq \G$ is unique; see Theorem~\ref{Tembed}.)
Thus, our suspected ``most primitive setting'' for proving stochastic
inequalities, of working with $\G_{strong}$-valued random variables, is
actually \textit{equivalent} to working with $\G_\circ$-valued random
variables.
This equivalence parallels Proposition~\ref{Psubset} (drawn from recent
joint work~\cite{KR2}), which gives a similar equivalence for
bi-invariant metric semigroups vs.\ monoids.

\section{The Hoffmann-J{\o}rgensen inequality for strongly left-invariant
metric semigroups}

The first inequality that we extend to strongly left/right-invariant
metric semigroups is the Hoffmann-J{\o}rgensen inequality, which is used
to bound sums of independent random variables. We refer the reader to
e.g.~\cite{HM} for a detailed history of the inequality.

\subsection{The inequality over strongly left/right-invariant metric,
following previous variants}

Here we will present a few versions of the above inequality from the
literature, ending with a general, unifying variant that holds over all
bi-invariant metric semigroups, before extending it (and hence the
preceding variants) to the more general setting of strongly
left/right-invariant metric semigroups. We start with a version found in
the monograph by Ledoux--Talagrand, where the authors attribute the
result to Hoffmann-J{\o}rgensen~\cite{HJ} (see also Kahane~\cite{Kah}).

\begin{theorem}[{\cite[Proposition 6.7]{LT}}]\label{Thjlt}
Suppose $\bb$ is a separable Banach space, and $(\Omega, \mathscr{A},
\mu)$ is a probability space with $X_1, \dots, X_n \in L^0(\Omega,\bb)$
independent random variables.
For $1 \leqslant j \leqslant n$, define $S_j := X_1 + \cdots + X_j$ and
$U_n := \max_{1 \leqslant j \leqslant n} \| S_j \|$. Then
\[
\bp{U_n > 3t+s} \leqslant \bp{U_n > t}^2 +
\bp{\max_{1 \leqslant j \leqslant n} \| X_j \| > s},
\qquad \forall s,t \in (0,\infty).
\]
\end{theorem}

Theorem~\ref{Thjlt} has led to several strengthenings and variants,
including by Johnson--Schechtman~\cite{JS}, Klass--Nowicki~\cite{KN}, and
Hitczenko and Montgomery-Smith~\cite{HM}. This last variant says:

\begin{theorem}[{\cite[Theorem 1]{HM}}]\label{Thm}
(Notation as in Theorem~\ref{Thjlt}.) For all $K \in \mathbb{Z}_{>0}$ and
$s,t \in (0,\infty)$,
\[
\bp{U_n > 2Kt + (K-1)s} \leqslant \frac{1}{K!} \left( \frac{\bp{U_n >
t}}{\bp{U_n \leqslant t}} \right)^K + \bp{\max_{1 \leqslant j \leqslant
n} \| X_j \| > s}.
\]
\end{theorem}

Notice from their statements that neither of Theorems~\ref{Thjlt}
and~\ref{Thm} immediately follow from the other. They were both
simultaneously extended in recent work \cite{KR1}, by the next variant
that we state. To do so, we require some notation.

\begin{definition}
Suppose $(\G, d_\G)$ is a separable semigroup with a left-invariant
metric, with Borel $\sigma$-algebra $\mathscr{B}_\G$. Given integers $1
\leqslant j \leqslant n$ and random variables $X_1, \dots, X_n : (\Omega,
\mathscr{A}, \mu) \to (\G, \mathscr{B}_\G)$, define
\begin{equation}\label{ESMX}
S_j(\omega) := X_1(\omega) \cdots X_j(\omega), \qquad
M_j(\omega) := \max_{1 \leqslant i \leqslant j} d_\G(z_0, z_0
X_i(\omega)),
\end{equation}

\noindent where $z_0 \in \G$ is arbitrary.
\end{definition}

\begin{remark}
Clearly, $M_j$ is independent of $z_0 \in \G$ in any strongly
left/right-invariant metric semigroup, hence -- as observed above -- in
any left/right-invariant metric monoid. We explain presently why this
implies the same fact over every bi-invariant metric semigroup as
in~\cite{KR1}.
\end{remark}

With this notation at hand, we state the next -- and rather general --
variant of the Hoffmann-J{\o}rgensen inequality:

\begin{theorem}[{\cite[Theorem~A]{KR1}}]\label{Thj}
Suppose $(\G, d_\G)$ is a separable bi-invariant metric semigroup, $z_0,
z_1 \in \G$ are fixed, and $X_1, \dots$, $X_n \in L^0(\Omega,\G)$ are
independent.
Also fix integers $0 < k, n_1, \dots, n_k \in \mathbb{Z}$ and nonnegative
scalars $t_1, \dots, t_k, s \in [0,\infty)$, and define
\begin{equation}
U_n := \max_{1 \leqslant j \leqslant n} d_\G(z_1, z_0 S_j), \ \
I_0 := \{ 1 \leqslant i \leqslant k : \bp{U_n \leqslant t_i}^{n_i -
\delta_{i1}} \leqslant \frac{1}{n_i!} \},
\end{equation}

\noindent where $\delta_{i1}$ denotes the Kronecker delta.
Now if $\sum_{i=1}^k n_i \leqslant n+1$, then
\begin{align}\label{Ehj}
&\ \bp{U_n > (2 n_1 - 1) t_1 + 2 \sum_{i=2}^k n_i t_i + \left(
\sum_{i=1}^k n_i - 1 \right) s}\\
\leqslant &\ \bp{U_n \leqslant t_1}^{{\bf 1}_{1 \notin I_0}} \prod_{i \in
I_0} \bp{U_n > t_i}^{n_i} \prod_{i \notin I_0} \frac{1}{n_i!} \left(
\frac{\bp{U_n > t_i}}{\bp{U_n \leqslant t_i}} \right)^{n_i}\notag\\
&\ + \bp{M_n > s}.\notag
\end{align}

More generally, define
\begin{align*}
&\ K := \sum_{i=1}^k n_i, \qquad Y_j := d_\G(z_0, z_0 X_j),\\
&\ Y_{(1)} := \min(Y_1, \dots, Y_n), \quad \dots, \quad
Y_{(n)} := \max(Y_1, \dots, Y_n),
\end{align*}

\noindent so that $Y_{(j)}$ are the order statistics of the $Y_j$. Then
the above inequality can be strengthened by replacing $\bp{M_n > s}$ by
\[
\bp{\sum_{j=n-K+2}^n Y_{(j)} > (K-1)s}.
\]
\end{theorem}

It is this result whose setting we weaken, to obtain the main result of
this section:

\begin{theorem}\label{T1}
Theorem~\ref{Thj} holds more generally, over every strongly
left/right-invariant metric semigroup (equivalently by
Theorem~\ref{Tembed}, over every left/right-invariant metric monoid).
\end{theorem}

As explained in~\cite{KR1}, Theorem~\ref{T1} specializes to
Theorem~\ref{Thjlt} by setting
\[
\G = \bb, \quad z_0 = z_1 = 0, \quad k = 2, \quad n_1 = n_2 = 1, \quad
t_1 = t_2 = t,
\]
so that $I_0 = \{ 1, 2 \}$. It also specializes to Theorem~\ref{Thm}, by
setting $\G = \bb, \ z_0 = z_1 = 0, \ k = 1, \ n_1 = K, \ t_1 = t$.
But moreover, Theorem~\ref{T1} specializes to Theorem~\ref{Thj} itself
(which implied the other two results above) via Proposition~\ref{Psubset}
below: every bi-invariant metric semigroup embeds into a bi-invariant
metric monoid, hence is a strongly left/right-invariant metric semigroup.

\subsection{Additional previous variants}

The proof of Theorem~\ref{T1} is deferred to the next subsection,
together with the fact that it implies Theorem~\ref{Thj}.
For now, we elaborate in some detail the connection of Theorems~\ref{Thj}
and~\ref{T1} to (special cases of) results by Johnson--Schechtman
\cite{JS} and Klass--Nowicki \cite{KN}. First, Johnson--Schechtman showed
by iterating the ``original'' Hoffmann-J{\o}rgensen inequality:

\begin{theorem}[{\cite[Lemma 6]{JS}}]\label{Tjs}
If $X_1, \dots, X_n$ are independent non-negative real-valued random
variables, and $z_0 = z_1 = 0$ (so that $S_n = \| S_n \| = U_n$), then
\[
\bp{U_n > (2k-1) t} \leq \bp{M_n > t} + \bp{U_n > t}^k, \qquad \forall t
\in (0,\infty), \ k \in \mathbb{Z}_{>0}.
\]
\end{theorem}

Theorem~\ref{Thj} (and hence Theorem~\ref{T1}) implies a weaker form of
this result, in which $\bp{U_n > t}$ is replaced by $\bp{U_n > t/2}$. To
see this, set
\[
\G = \mathbb{R}, \ \ z_0 = z_1 = 0, \ \ n_1 = \cdots = n_k = 1, \ \ t_1 =
s = t, \ t_2 = \cdots = t_k = t/2,
\]
so that $I_0 = \{ 1, \dots, k \}$. Also note that Theorem~\ref{Tjs} holds
for non-negative variables, so $\G$ is necessarily $\mathbb{R}$, while
Theorem~\ref{T1} holds for all $\G$ -- including possibly negative
real-valued random variables -- and its assertion in this generality is
slightly weaker.

Second, Hitczenko and Montgomery-Smith term their result
(Theorem~\ref{Thm}) the \textit{Klass--Nowicki inequality}. This is
because Klass--Nowicki had previously shown a slightly different -- but
related -- result to Theorem~\ref{Thm}:

\begin{theorem}[{\cite[Theorem 1.1]{KN}}]\label{Tkn}
Suppose $X_1, \dots, X_n$ are independent Banach space-valued random
variables, and we define $S_n, U_n$ as above with $z_0 = z_1 = 0$. Also
set $\lambda := \bp{U_n \geq 1}$ and suppose $\lambda < 1$. If the $X_j$
are non-negative, then
\begin{equation}\label{Ekn1}
\bp{ \| S_n \| \geq k + Y_{(n)} + \cdots + Y_{(n-k+2)} }\leq \frac{1}{k!}
[ n (1 - \sqrt[n]{1-\lambda}) ]^k, \qquad \forall k \geq 1,
\end{equation}

\noindent whereas if the $X_j$ are symmetric, then
\begin{equation}\label{Ekn2}
\bp{ U_n \geq k + Y_{(n)} + \cdots + Y_{(n-k+2)} } \leq
\frac{2^{k-1}}{k!} [ n (1 - \sqrt[n]{1-\lambda}) ]^k,
\qquad \forall k \geq 1.
\end{equation}
\end{theorem}

Note that the left-hand expression in~\eqref{Ekn1} is bounded above by
the one in~\eqref{Ekn2}. These results relate to Theorems~\ref{Thj}
and~\ref{T1} with $t = s = 1$ -- even when the $X_j$ need not be
non-negative or symmetric -- as follows:
\begin{align*}
&\ \bp{ U_n > 2kt + (k-1)s }\\
\leq &\ \bp{ U_n > kt + (k-1)s }\\
\leq &\ \bp{ U_n > k + Y_{(n)} + \cdots + Y_{(n-k+2)} } + \bp{ Y_{(n)} +
\cdots + Y_{(n-k+2)} > (k-1) }\\
\leq &\ \frac{2^{k-1}}{k!} [ n (1 - \sqrt[n]{1-\lambda}) ]^k
+ \bp{ Y_{(n)} + \cdots + Y_{(n-k+2)} > (k-1) },
\end{align*}
where the final inequality uses~\eqref{Ekn2}. Now if the upper bound in
the final inequality was missing the factor of $2^{k-1}$, then we
\textbf{claim} that this bound is at most $\frac{1}{k!}
(\lambda/(1-\lambda))^k$, which is the precise factor in
Theorem~\ref{Thm} and hence in a suitable specialization of
Theorem~\ref{Thj}. Thus, Theorem~\ref{Tkn} would imply a very similar
result to Theorem~\ref{Thm} (strengthened to replace $\bp{ M_n > 1 }$ by
$\bp{ Y_{(n)} + \cdots + Y_{(n-k+2)} > (k-1) }$, as in \cite{HM} as well
as in Theorem~\ref{Thj}).

\begin{remark}
Here we quickly explain why the result obtained in the preceding
paragraph would be similar to Theorem~\ref{Thm} but not exactly
comparable. The preceding claim -- that the bound in~\eqref{Ekn1} is
lower than the bound in Theorem~\ref{Thm} -- is easily checked:
\[
\frac{1}{k!} [ n (1 - \sqrt[n]{1-\lambda}) ]^k \leq \frac{1}{k!} \left(
\frac{\lambda}{1-\lambda} \right)^k, \qquad \forall k \geq 1.
\]
This is because the inequality turns out to be equivalent to requiring $1
- \alpha \lambda \leq (1-\lambda)^\alpha$ for $\lambda \in [0,1)$ and
$\alpha = 1 + \frac{1}{n} \in [1,2]$, and this latter inequality holds by
the binomial series formula. However, the same inequality does not hold
when working with the bound in~\eqref{Ekn2} instead:
\[
\frac{2^{k-1}}{k!} [ n (1 - \sqrt[n]{1-\lambda}) ]^k \not\leq \frac{1}{k!}
\left( \frac{\lambda}{1-\lambda} \right)^k, \qquad \forall k > 1.
\]
Indeed, this non-inequality $\not\leq$ is equivalent to $>$, i.e., to the
assertion that
\[
2^{1 - 1/k} n (1-\lambda)(1 - \sqrt[n]{1-\lambda}) > \lambda,
\]
which does hold at very small values of $\lambda \in (0,1)$, since the
left-hand side is a power series in $\lambda$ with constant term zero and
linear term $2^{1-1/k} > 1$ (for $k>1$).
\end{remark}

\subsection{The proofs}

Before showing Theorem~\ref{T1}, let us prove that it implies
Theorem~\ref{Thj}: this is because every metric semigroup embeds into a
bi-invariant metric monoid, by the following result.

\begin{prop}[{\cite[\S 2.1]{KR3}}]\label{Psubset}
Every bi-invariant metric semigroup $\G$ is contained in a metric monoid.
More precisely, $\G$ contains at most one idempotent, which is
automatically a two-sided identity. Thus if $\G$ contains exactly one
idempotent then it is a metric monoid. Otherwise, $\G$ is the set of
non-identity elements in a metric monoid $\G_\circ$ (with identity $e$;
thus the smallest $\G_\circ \supseteq \G$ is unique), and $\G$ maps
isometrically into $\G_\circ$ with:
$d_{\G_\circ}(e,g) := d_\G(g,g^2)$ for all $g \in \G$.
\end{prop}

The proof is not hard (but omitted here; see~\cite{KR3}). Note that the
final statement about the extended metric being bi-invariant uses the
following calculation in any metric semigroup:
\begin{equation}\label{Einv}
d_\G(a,ba) = d_\G(ba, b^2 a) = d_\G(b, b^2) = d_\G(ab, a b^2) = d_\G(a,
ab), \qquad \forall a,b \in \G.
\end{equation}
Of course, the equalities in~\eqref{Einv} need not all hold in every
left/right-invariant metric monoid (such as in non-compact non-abelian
Lie groups $G$, which possess left-invariant but not necessarily
bi-invariant Riemannian metrics\footnote{For instance, by~\cite{Mil}, the
only connected Lie groups that admit a bi-invariant Riemannian metric are
of the form $K \times \mathbb{R}^n$, for $K$ a compact group and $n \geq
0$ an integer.}).

The other proof is that of our main result in this section; this is now
provided and includes fixing a small typo in \cite{KR1} itself.

\begin{proof}[Proof of Theorem~\ref{T1}]
The idea is to closely follow the proof of Theorem~\ref{Thj}, making the
necessary adjustments along the way because now $d_\G$ is only strongly
left-invariant. The first change is that the assertion that
$Y_j := d_\G(z_0, z_0 X_j)$,
$M_j(\omega) := \max_{1 \leq i \leq j} Y_j$
do not depend on $z_0$,
which was shown in~\cite{KR1} using~\eqref{Einv} via bi-invariance,
now follows instead from the strong left-invariance of $d_\G$.
The next change involves avoiding the use of~\eqref{Einv} in Step~2 of
the proof, in the long calculation on pp.~4106 of \cite{KR1}, where we
computed:
\begin{align*}
d_\G(z_0 S_{m-1}(\omega), z_0 S_m(\omega)) =
d_\G(z_0 S_{m-1}(\omega), z_0 S_{m-1}(\omega) \cdot X_m(\omega))
= &\ d_\G(z_0, z_0 X_m(\omega))\\
= &\ Y_m(\omega).
\end{align*}

\noindent Instead, this calculation now holds using merely the strong
left-invariance of $d_\G$.

A small remark here concerns an argument in \cite[Proof of Theorem~1]{HM}
over Banach spaces, that was black-boxed in~\cite{KR1}:
\[
\sup_{j,k \leq n} \| S_k - S_j \| \leq 2 \sup_j \| S_j \| = 2 U_n.
\]
This was used in~\cite{KR1} for bi-invariant metric semigroups; as we now
note, one requires only the (not even strong) left-invariance of $d_\G$,
and not an identity in $\G$, since by the triangle inequality,
\[
\sup_{j,k \leq n} d_\G(S_k, S_j) = \max_{1 \leq j,k \leq n} d_\G(z_0 S_k,
z_0 S_j) \leq 2 \max_{1 \leq j \leq n} d_\G(z_1, z_0 S_j) = 2 U_n. 
\]

The final point here is a modification that in fact applies to the
\textit{original} proof itself in~\cite{KR1} -- we fix a small typo
there. Namely, in \cite[Equation~(9)]{KR1}, the definition of
$p_{\beta,t}$ should be modified to use $0 < j$ instead of $0 \leq j$.
Thus we need to define and work with
\begin{equation}
p_{\beta,t} := \bp{d_\G(z_1, z_0 S_\beta) > t \geq
d_\G(z_1, z_0 S_j)\ \forall 0 < j < \beta}, \qquad 0 < \beta \leq n.
\end{equation}
\end{proof}

\begin{remark}
We leave it to the interested reader to work out the case of strongly
right-invariant metric semigroups, for the results in this section and
the next. For this, when working only over semigroups, the results and
proofs involve the modified quantities
\[
S_j(\omega) := X_j(\omega) X_{j-1}(\omega) \cdots X_1(\omega), \qquad
U_n := \max_{1 \leq j \leq n}d_\G(z_1, S_j z_0), \qquad Y_j := d_\G(z_0,
X_j z_0).
\]
However, if $\G$ is in fact a group then one does not need to prove the
``strongly right-invariant'' analogues (of all results in this paper)
separately, since it follows from the left-invariant results in
Theorem~\ref{T1}. This is because of the bijection between the sets of left- and
right-invariant metrics on a group $G$:
$d_G^L(a,b) \ \longleftrightarrow \ d_G^R(a^{-1}, b^{-1})$.
The same holds if $G$ is abelian, since any left/right-invariant metric
$d_G$ is automatically bi-invariant, and the results in this work reduce
to those in previous works \cite{KR1,KR2,KR3}.
\end{remark}

\section{Decreasing rearrangements and applications}

Here we write down a couple of applications of the Hoffmann-J{\o}rgensen
inequality. The first shows that controlling the behavior of the
independent variables $X_j$ is the same as controlling $S_n$ or $U_n$ --
but now over all strongly left/right-invariant semigroups (equivalently
by Theorem~\ref{Tembed}, over left/right-invariant metric monoids):

\begin{theorem}\label{Thj2}
Suppose $A \subset \mathbb{Z}_{>0}$ is either $\mathbb{Z}_{>0}$ or $\{ 1,
\dots, N \}$ for some $N \in \mathbb{Z}_{>0}$. Suppose $(\G, d_\G)$ is a
separable strongly left-invariant metric semigroup, $z_0, z_1 \in \G$,
and the variables $X_n \in L^0(\Omega,\G)$ are independent for all $n \in
A$. If $\sup_{n \in A} d_\G(z_1, z_0 S_n) < \infty$ almost surely, then
for all $p \in (0,\infty)$,
\[
\mathbb{E}_\mu \left[ \sup_{n \in A} d_\G(z_0, z_0 X_n)^p \right] <
\infty \quad \Longleftrightarrow \quad \mathbb{E}_\mu \left[ \sup_{n \in
A} d_\G(z_1, z_0 S_n)^p \right] < \infty.
\]
\end{theorem}

This extends \cite[Theorem~A]{KR2} and originally \cite[Theorem~3.1]{HJ}
to the primitive setting of strongly left-invariant metric semigroups.
The proof of this and the next few results use the theory of quantile
functions / decreasing rearrangements:

\begin{definition}
Say $(\G, d_\G)$ is a strongly left-invariant metric semigroup, and $X :
(\Omega, \mathscr{A}, \mu) \to (\G,\mathscr{B}_\G)$. The
\textit{decreasing} (or \textit{non-increasing}) \textit{rearrangement}
of $X$ is the right-continuous inverse $X^*$ of the function $t \mapsto
\bp{d_\G(z_0, z_0 X) > t}$, for any $z_0 \in \G$. In other words, $X^*$
is the real-valued random variable defined on $[0,1]$ with the Lebesgue
measure, as follows:
\[
X^*(t) := \sup \{ y \in [0,\infty) : \bp{d_\G(z_0, z_0 X) > y} > t \}.
\]
\end{definition}

Now the proof of Theorem~\ref{Thj2} uses the first assertion in
\cite[Proposition~6.8]{LT}, extended here from Banach spaces to strongly
left-invariant metric semigroups. It shows that controlling sums of
$\G$-valued $L^p$ random variables in probability (i.e., in $L^0$) allows
us to control these sums in $L^p$ as well, for $p>0$.

\begin{prop}\label{Plt}
Suppose $(\G, d_\G)$ is a separable strongly left-invariant metric
semigroup, $p \in (0,\infty)$, and we have independent random variables
$X_1, \dots, X_n$ $\in L^p(\Omega,\G)$, i.e., $\mathbb{E}_\mu[ d_\G(z_0,
z_0 X_j)^p ] < \infty$ for all $j$ (and any choice of $z_0 \in \G$). Now
fix $z_0, z_1 \in \G$ and let $S_k, U_n, M_n$ be as in~\eqref{ESMX} and
Theorems~\ref{Thj}, \ref{T1}. Then,
\[
\mathbb{E}_\mu[U_n^p] \leq 2^{1 + 2p} (\mathbb{E}_\mu[M_n^p] +
U_n^*(2^{-1-2p})^p).
\]
\end{prop}

As this result was not proved in \cite{LT}, and given that we work here
in significantly greater generality, we write down a proof for
completeness.

\begin{proof}
The proof uses the following fact (which follows from the Pigeonhole
principle):
\begin{equation}\label{Epigeon}
0 \leq a \leq \sum_{j=1}^k w_j a_j,\ 0 \leq a_j, w_j\ \forall j \quad
\implies \quad a^p \leq (\sum_j w_j)^p \sum_{j=1}^k a_j^p\ \ \forall p
\in (0, \infty). 
\end{equation}

Now to prove the result, note that $d_\G(z_1, z_0 S_k) \leq d_\G(z_1,
z_0) + \sum_{j=1}^k d_\G(z_0, z_0 X_j)$ for all $1 \leq k \leq n$, by the
triangle inequality and the strong left-invariance of $d_\G$.
Using~\eqref{Epigeon} with all $w_j = 1$,
\[
U_n^p \leq \left( d_\G(z_1, z_0) + \sum_{j=1}^n d_\G(z_0, z_0 X_j)
\right)^p \leq (n+1)^p \left( d_\G(z_1, z_0)^p + \sum_{j=1}^n d_\G(z_0,
z_0 X_j)^p \right).
\]

\noindent Hence $U_n \in L^p(\Omega,\G)$ if all $X_j \in L^p(\Omega,\G)$.
Now recall that if $Z : (\Omega, \mathscr{A}, \mu) \to [0,\infty)$, then
\begin{equation}\label{Edr}
\mathbb{E}_\mu[Z^\alpha] = \alpha \int_0^\infty t^{\alpha - 1} \bp{Z>t}\
dt, \qquad \forall \alpha > 0.
\end{equation}

\noindent Apply Equation~\eqref{Edr} with $Z = U_n$, $\alpha = p$, and $t
\leadsto 4t$. Thus the integral converges, and we compute for any fixed
$u \geq 0$ (using the Hoffmann-J{\o}rgensen inequality in
Theorem~\ref{T1} with $k=2, n_1 = n_2 = 1, t_1 = t_2 = s = t$, so that
$I_0 = \{ 1, 2 \}$):
\begin{align*}
\mathbb{E}_\mu[U_n^p] = &\ p \int_0^\infty (4t)^{p-1} \bp{U_n > 4t}\
d (4t) = p 4^p \int_0^\infty t^{p-1} \bp{U_n > 4t}\ dt\\
= &\ p 4^p \left( \int_0^u + \int_u^\infty \right) t^{p-1} \bp{U_n > 4t}\
dt\\
\leq &\ 4^p u^p + p 4^p \int_u^\infty t^{p-1} [\bp{U_n > t}^2 + \bp{M_n >
t}]\ dt\\
\leq &\ 4^p u^p + p 4^p \int_u^\infty t^{p-1} \bp{U_n > t}^2\ dt + 4^p
\mathbb{E}_\mu[M_n^p],
\end{align*}

\noindent where the final inequality follows from~\eqref{Edr} with $Z =
M_n \in L^p(\Omega,\mathbb{R})$, $\alpha = p$. Now suppose $u >
U_n^*(2^{-1-2p})$; then the outstanding integrand can be bounded above
via
\[
\bp{U_n > t}^2 \leq \bp{U_n > u} \bp{U_n > t} \leq 2^{-1-2p} \bp{U_n >
t}.
\]

\noindent Continuing with the above calculations,
\[
\mathbb{E}_\mu[U_n^p] \leq (4u)^p + 4^p \mathbb{E}_\mu[M_n^p] + 4^p \cdot
2^{-1-2p} \cdot \int_0^\infty p t^{p-1} \bp{U_n > t}\ dt = (4u)^p +
4^p \mathbb{E}_\mu[M_n^p] + 2^{-1} \mathbb{E}_\mu[U_n^p],
\]

\noindent by a third application of~\eqref{Edr} with $Z = U_n$, $\alpha =
p$. Since this inequality holds for all $u > U_n^*(2^{-1-2p})$, the
desired claim follows.
\end{proof}

With Proposition~\ref{Plt} at hand, one shows Theorem~\ref{Thj2} by
closely following the proof of \cite[Theorem~A]{KR2}. Here we only point
out the modifications required: in addition to using
Proposition~\ref{Plt} over strongly left-invariant metric semigroups, the
only other update is that in showing that
\[
d_\G(z_0, z_0 X_n) \leqslant d_\G(z_1, z_0 S_{n-1}) + d_\G(z_1, z_0 S_n),
\]
we now use the strong left-invariance rather than the bi-invariance of
the metric $d_\G$. \qed\medskip

The other result that we discuss and extend here (from bi-invariant to
strongly left-invariant metric semigroups), again uses the
Hoffmann-J{\o}rgensen inequality to relate the $L^p$-norm of $U_n$ to its
tail distribution (using $U_n^*$).

\begin{prop}\label{Lbounds}
There exists a universal positive constant $c_1$ such that for any $0
\leq t \leq s \leq 1/2$, any separable strongly left-invariant metric
semigroup $(\G, d_\G)$ with elements $z_0, z_1$, and any sequence of
independent $\G$-valued random variables $X_1, \dots, X_n$,
\[
U_n^*(t) \leq c_1 \frac{\log(1/t)}{\max \{ \log(1/s), \log \log(4/t)
\}} ( U_n^*(s) + M_n^*(t/2)),
\]

\noindent with $U_n := \max_{1 \leqslant j \leqslant n} d_\G(z_1, z_0
S_j)$ and $M_n := \max_{1 \leq j \leq n} d_\G(z_0, z_0 X_j)$ as above.
\end{prop}

We again omit the proof, as it is the same as that of
\cite[Lemma~3.1]{KR2}, via the proof of \cite[Corollary~1]{HM}. \qed

\section{Additional probability inequalities and L\'evy's equivalence}

We next extend several other probability inequalities to strongly
left/right-invariant metric semigroups (in particular, they now become
applicable to all Lie groups), from their previous avatars over either
bi-invariant metric monoids, or in one case, Banach spaces. These
inequalities also help extend the following result:

\begin{theorem}[L\'evy's equivalence]\label{Tlevy}
Suppose $(\G, d_\G)$ is a complete separable strongly left-invariant
metric semigroup, $X_n : (\Omega, \mathscr{A}, \mu) \to (\G,
\mathscr{B}_\G)$ are independent, $X \in L^0(\Omega, \G)$, and $S_n$ is
as in \eqref{ESMX}. Then
\[
S_n \longrightarrow X\ a.s.~\mathbb{P}_\mu \quad \Longleftrightarrow
\quad S_n \conv{P} X.
\]

\noindent If instead the sequence $S_n$ does not converge in the above
manner, then it diverges almost surely.
\end{theorem}

Via Proposition~\ref{Psubset}, this result simultaneously extends
\cite[Theorem~2.1]{KR2} for $\G$ a complete separable bi-invariant
metric semigroup -- which in turn had extended variants for $\G$ a
separable Banach space by
It$\hat{\mbox{o}}$--Nisio \cite[Theorem 3.1]{IN} and by
Hoffmann-J{\o}rgensen and Pisier \cite[Lemma 1.2]{HJP} --
as well as Tortrat's result in~\cite{To2} for $\G$ a complete separable
metric group.

\begin{remark}
We claim that the setting in Theorem~\ref{Tlevy} is strictly more general
than even \cite[Theorem~2.1]{KR2} -- which itself was more general than
the preceding variants. This follows by considering examples of Lie
groups with left-invariant but not bi-invariant Riemannian metric (by the
results in \cite{Mil}, as discussed above).
\end{remark}

The proof of Theorem~\ref{Tlevy} employs the Ottaviani--Skorohod
inequality:

\begin{prop}[Ottaviani--Skorohod inequality]\label{Pos}
Suppose $(\G, d_\G)$ is a separable strongly left-invariant metric
semigroup, and $X_1, \dots, X_n : \Omega \to \G$ are independent. Fix $0
< \alpha, \beta \in \mathbb{R}$. Then for all $z_0, z_1 \in \G$,
\[
\bp{\max_{1 \leq k \leq n} d_\G(z_1, z_0 S_k) \geq \alpha  +\beta}
\cdot \min_{1 \leq k \leq n} \bp{ d_\G(S_k, S_n) \leq \beta} \leq
\bp{ d_\G(z_1, z_0 S_n) \geq \alpha}.
\]
\end{prop}

This is a special case of the Mogul'skii inequalities (setting $m=1, a =
\alpha + \beta, b = \beta$ in the next result) -- in~\cite{KR2} for
bi-invariant metric semigroups, but also more generally in the present
setting:

\begin{prop}[Mogul'skii inequalities]\label{Pmog}
Suppose $(\G, d_\G)$ is a separable strongly left-invariant metric
semigroup, $z_0, z_1 \in \G$, $a,b,c \in [0,\infty)$, and $X_1, \dots,
X_n \in L^0(\Omega,\G)$ are independent. If $1 \leq m \leq n$ in
$\mathbb{Z}$, then:
\begin{align*}
& \bp{\min_{m \leq k \leq n} d_\G(z_1, z_0 S_k) \leq a} \cdot \min_{m
\leq k \leq n} \bp{ d_\G(S_k, S_n) \leq b} \leq \bp{ d_\G(z_1, z_0 S_n)
\leq a + b},\\
& \bp{\max_{m \leq k \leq n} d_\G(z_1, z_0 S_k) \geq a} \cdot \min_{m
\leq k \leq n} \bp{ d_\G(S_k, S_n) \leq b} \leq \bp{ d_\G(z_1, z_0 S_n)
\geq a - b}.
\end{align*}
\end{prop}

It is the formulation of the above results in this greater generality
that is of note; we omit all three of the proofs here (as in the
preceding section), as they go through verbatim except for a common
workaround needed in one step in each proof. Namely: one needs to equate
$d_\G(z_0 S_k, z_0 S_n)$ for $k \leq n$ with $d_\G(z_0, z_0 X_{k+1}
\cdots X_n)$ -- and this follows from the strong left-invariance of $\G$.
\qed\medskip

The next result is the L\'evy--Ottaviani inequality, now extended
from reals to our general setting:

\begin{prop}[L\'evy--Ottaviani inequality]\label{Pbil}
Suppose $(\G, d_\G)$ is a separable strongly left-invariant metric
semigroup, $z_0, z_1 \in \G$, and $X_1, \dots, X_n \in L^0(\Omega,\G)$
are independent. Given $a \geq 0$, define:
\[
U_n := \max_{1 \leq k \leq n} d_\G(z_1, z_0 S_k), \qquad
p_a := \max_{1 \leq k \leq n} \bp{d_\G(z_1, z_0 S_k) > a},
\]

\noindent where $S_k := X_1 \cdots X_k$ for all $k$, as above. Then for
all $l \geq 2$ and $a_1, \dots, a_l \geq 0$,
\[
\bp{U_n > a_1 + \dots + a_l} \leq \sum_{i=2}^l p_{a_i} + p'_l,
\]

\noindent where $p'_l := p_{a_1}$ if $l$ is odd, and $p'_l := \max_{1
\leq k \leq n} \bp{(d_\G(S_k, S_n) > a_1}$ if $l$ is even.
\end{prop}

This result generalizes \cite[Theorem 22.5]{Bi1}, both in its setting and
its statement. To see this, set
\[
l=3, \qquad a_1 = a_2 = a_3 = \alpha, \qquad \G = (\mathbb{R},+), \qquad z_0
= z_1 = 0.
\]
Also note that the result is false if $l=1$ and $p'_1 = p_{a_1}$, since
$\bp{d_\G(z_1, z_0 S_k) > a_1} \leq \bp{U_n > a_1}$ for all $k$.

\begin{proof}
We write down a proof in this general setting, to indicate where one
needs to use strong left-invariance (as opposed to any group or Banach
space structure).
Define the stopping time $\tau : \Omega \to \{ 1, \dots, n \}$ via:
\[
\tau = \inf \{ k \in [1,n] \cap \mathbb{Z}\ :\ d_\G(z_1, z_0 S_k) > a_1 +
\dots + a_l \}.
\]

\noindent It is now not hard to show that:
\begin{align*}
\bp{U_n > \sum_{i=1}^l a_i} \leq &\ \bp{d_\G(z_1, z_0 S_n) > a_l}
+ \sum_{k=1}^{n-1} \bp{\tau = k, d_\G(z_1, z_0 S_n) \leq a_l}\\
\leq &\ p_{a_l} + \sum_{k=1}^{n-1} \bp{\tau = k, d_\G(S_k, S_n) >
\sum_{i=1}^{l-1} a_i}\\
= &\ p_{a_l} + \sum_{k=1}^{n-1} \bp{\tau = k} \bp{d_\G(S_k, S_n) >
\sum_{i=1}^{l-1} a_i},
\end{align*}

\noindent where the final equality uses the independence of $\tau = k$
from the behavior of $d_\G(S_k, S_n) = d_\G(z_0, z_0 X_{k+1} \cdots
X_n)$; this last uses the strong left-invariance of $d_\G$. There are now
two cases:
\begin{itemize}
\item If $l=2$, then the last probability (inside the sum) is bounded
above by $p'_l$ for $l=2$, and the result follows since $\sum_{k=1}^{n-1}
\bp{\tau = k} \leq 1$.

\item If $l > 2$ and $d_\G(S_k, S_n) > \sum_{i=1}^{l-1} a_i$, then either
$d_\G(z_1, z_0 S_k) > a_{l-1}$, or $d_\G(z_1, z_0 S_n)$ $> \sum_{i=1}^{l-2}
a_i$. By definition, the first event has probability at most
$p_{a_{l-1}}$, while the probability of the latter event is analyzed in
two sub-cases: first, if $l=3$, then it is dominated by $p'_l = p_{a_1}$,
which proves the result as follows:
\[
\bp{U_n > \sum_{i=1}^3 a_i} \leq p_{a_3} + \sum_{k=1}^{n-1} \bp{\tau =
k}(p_{a_2} + p_{a_1}) \leq p_{a_3} + 1 \cdot (p_{a_2} + p'_l).
\]

Next, if $l > 3$, then we prove the result by using induction and the
above analysis for the base cases of $l=2,3$: the second event is
dominated by $U_n > \sum_{i=1}^{l-2} a_i$. Hence by the induction
hypothesis for $l-2$,
\begin{align*}
\bp{U_n > \sum_{i=1}^l a_i} \leq &\ p_{a_l} + \sum_{k=1}^{n-1}
\bp{\tau = k}(p_{a_{l-1}} + \bp{U_n > \sum_{i=1}^{l-2} a_i})\\
\leq &\ p_{a_l} + 1 \cdot (p_{a_{l-1}} + \sum_{i=2}^{l-2} p_{a_i} +
p'_{l-2}),
\end{align*}
and we are done since $p'_l$ depends only on the parity of $l$. \qedhere
\end{itemize}
\end{proof}

\begin{remark}
We quickly discuss additional results in the present, general setting.
First, several other results in the literature were shown in
\cite{KR1,KR2,KR3} for random variables valued in abelian metric
semigroups $\G_{ab}$.
These results automatically extend to the present setting (since the
metric in $\G_{ab}$ is bi-invariant, hence strongly left-invariant).

Second, further results were shown in \textit{loc.\ cit.}, over all
semigroups with a bi-invariant metric. Given the above results in the
paper, we expect that these latter results should hold for all strongly
left/right-invariant metric semigroups -- for instance, the results on
the L\'evy property found in \cite[\S 4]{HM}. We do not pursue this
investigation in the present work, leaving it to the interested reader.

Third, we leave the parallel formulation and proofs of the strongly
right-invariant analogues of these results (discussed here) and of the
ones studied in this section, to the interested reader.
\end{remark}

\section{Embedding (strong) left-invariant semigroups in monoids,
following Malcev}

This concluding section is not concerned with results in probability, but
with the framework underlying this entire work: strongly left-invariant
metric semigroups. In previous work and in the results above, we have
seen random variables take values in three different competing structures
(which are each more primitive than normed linear spaces, or even
groups):
\begin{enumerate}[(i)]
\item bi-invariant metric semigroups/monoids,
\item (strongly) left-invariant metric monoids,
\item strongly left-invariant metric semigroups.
\end{enumerate}

By Proposition~\ref{Psubset}, every bi-invariant metric semigroup embeds
isometrically and homomorphically into a left-invariant (in fact,
bi-invariant) metric monoid -- which is automatically a strongly left-
and right-invariant monoid. Thus, (i)~is a strictly more restrictive
notion than~(ii), since -- as explained above using~\cite{Mil} -- there
exist Lie groups that are equipped with (strongly) left-invariant metrics
that are not bi-invariant.

The other comparison question is between strongly left-invariant metric
(ii)~monoids and (iii)~semigroups (or right-invariant). If indeed these
notions are the same, then this entire paper could have equivalently been
written for left-invariant metric monoid-valued random variables. And
indeed, we now show this to be the case:

\begin{theorem}\label{Tembed}
Let $\G$ be a semigroup with a left-invariant metric $d_\G$.
\begin{enumerate}
\item Every idempotent in $\G$ is a left-identity. Thus, every
right-identity is a two-sided identity. (The converse statements are
obvious.)

\item Suppose $\G$ has no two-sided/right identity. Then $\G$ embeds
isometrically and homomorphically inside a left-invariant metric monoid
(onto the non-identity elements) if and only if $d_\G$ is strongly
left-invariant. In particular, such a semigroup $\G$ has no idempotents.
\end{enumerate}
\end{theorem}

We leave it to the interested reader to formulate and prove the (obvious)
counterpart for right-invariant metric semigroups $\G$. Also, this
situation is parallel to the case of bi-invariant metric semigroups,
which always embed inside a bi-invariant metric monoid by
Proposition~\ref{Psubset}.

\begin{proof}
First, if $e^2 = e \in \G$ then $d_\G(eg,g) = d_\G(e^2g,eg) = 0$, so
$eg=g$ for all $g \in \G$, as claimed. In particular, since a
right-identity is an idempotent, it is automatically a two-sided
identity.

Next, if $\G$ embeds isometrically (and strictly) inside a left-invariant
metric monoid, say $(\G_\circ,e,d)$, then $\G$ is strongly left-invariant
(as repeatedly seen above). Moreover, $\G$ has no idempotent, for if $g$
is any idempotent in $\G_\circ$ then $d(e,g) = d(g,g^2) = 0$, so $g = e
\not\in \G$.

Conversely, if $d_\G$ is strongly left-invariant, define $\G_\circ := \G
\sqcup \{ e \}$, and extend the semigroup product to $\G_\circ$ via: $e
\cdot g = g \cdot e := g$ for all $g \in \G_\circ$. Also extend the
metric $d_\G$ to a symmetric map $d$ on all of $\G_\circ \times \G_\circ$
via: $d \equiv d_\G$ on $\G \times \G$, $d(e,e) := 0$, and
\[
d(e,g) = d(g,e) := d_\G(g, g^2)\ \forall g \in \G.
\]
We claim this last expression is positive. If not, then $g^2 = g$, so $g$
is a left-identity by the preceding part. Moreover, by strong
left-invariance,
\[
d_\G(a,ag) = d_\G(g,g^2) = 0\ \forall a \in \G,
\]
so $g$ is a two-sided identity, which is false.
We next claim that $d$ satisfies the triangle inequality -- for
which it suffices to work with $a,b \in \G$ and $e$. This too is verified
using strong left-invariance:
\begin{align*}
d(a,e) + d(e,b) = &\ d_\G(a^2,a) + d_\G(b,b^2) = d_\G(ba,b) + d_\G(b,b^2)
\geq d_\G(ba,b^2) = d(a,b),\\
d(a,b) + d(e,b) = &\ d_\G(ba,b^2) + d_\G(b,b^2) \geq d_\G(ba,b) =
d_\G(a^2,a) = d(a,e).
\end{align*}
Hence $d$ is a metric. That $d$ is left-invariant follows from the strong
left-invariance of $d_\G$.
\end{proof}

Following the comparisons between the primitive settings (i)--(iii) at the
start of this section, a final comparison to make is between (ii)=(iii)
(shown above) and (iv)~left-invariant metric semigroups. Namely:
\textit{Can the word ``strong'' be removed from Theorem~\ref{Tembed}(2)?}
While the proofs above use strong left-invariance, it is not clear if
this itself follows (or not) from ``usual'' left-invariance in any
semigroup, as it does in all monoids. Thus, if the question has a
positive answer, this paper could even have been written over all
(iv)~left-invariant metric semigroups. However, we now show this is not
true -- even for finitely generated complete metric semigroups, cf.\
Theorem~\ref{Tlevy} -- by providing a counterexample to the question
above (also alluded to in the opening section):

\begin{prop}\label{Pcounterex}
There exists a countable, discrete (hence complete) left-invariant metric
semigroup $\G$ with two generators, which is not a monoid but contains an
idempotent.
\end{prop}

In particular, by Theorem~\ref{Tembed} $\G$ can never embed into a
left-invariant metric monoid, and $d_\G$ is not strongly left-invariant.

\begin{proof}
On the countable set $\G := \{ h^{n+1}, h^n g : n \in \mathbb{Z}_{\geq
0} \}$, define the operation
$h^n g^\varepsilon \cdot h^{n'} g^{\varepsilon'} := h^{n + n'}
g^{\varepsilon'}$.
Then $\G$ is a semigroup generated by $g,h$, with $g$ a left-identity
(and the only such), hence idempotent; but as $h^{n+1} g \neq h^{n+1}$
for each $n \geq 0$, $\G$ has no two-sided identity, so is not a monoid.

Next, define the Manhattan distance $d_\G : \G \times \G \to [0,\infty)$ via:
$d_\G(h^n g^\varepsilon, h^{n'} g^{\varepsilon'}) := |n-n'| +
|\varepsilon-\varepsilon'|$. Clearly, $d_\G$ is positive and symmetric.
The triangle inequality is verified by checking six easy cases; we show one
case here. For all $n' > 0$ and $n \geq 0$,
\[
d_\G(h^{n'}, h^n g) + d_\G(h^{n'}, h^m) = |n'-n|+1 + |n'-m| \geq |n-m|+1
= d_\G(h^m, h^n g).
\]
Thus $\G$ is a metric; it is easy to see that $(\G, d_\G)$ is countable
and discrete, hence complete.

The next step is to check that $d_\G$ is left-invariant, but not strongly
so. For the latter, we compute
\[
d_\G(h^n g, h^n g \cdot h^{n'}) = n'+1 \neq n' = d_\G(h^{n'}, h^{n'}
\cdot h^{n'}), \quad \forall n' > 0, \ n \geq 0,
\]
while the former is straightforward -- as a sample calculation, we have
\[
d_\G(h^n g^\varepsilon \cdot h^{n'} g^{\varepsilon'}, h^n g^\varepsilon
\cdot g) = n' + 1 - \varepsilon' = d_\G(h^{n'} g^{\varepsilon'}, g)
\]
for all $n'>0$, $\varepsilon' \in \{ 0, 1 \}$, and $h^n g^\varepsilon
\in \G$ (i.e.\ $(\varepsilon,n) \neq (0,0)$).
\end{proof}

\begin{remark}\label{Rpolymath}
Another remark, for completeness, is that given a left-invariant metric
\textit{group} $\G$, its metric $d_\G$ is 2-homogeneous -- i.e., $d_\G(e,
g^2) = 2 d_\G(e,g)$ for all $g \in \G$ -- if and only if $\G$ is abelian.
(In particular, $d_\G$ is bi-invariant.) This follows from the recent
Polymath project~\cite{Po}.
\end{remark}

For completeness, and given the above discussion in this section, we
conclude with a follow-up question to the above results:
\textit{Does every left-invariant metric monoid embed inside a
left-invariant metric group?}
The answer is in the negative for an even more general question: for
bi-invariant metric monoids to always embed inside some group! Namely, if
one takes $d_\G$ to be the discrete metric, then note that $d_\G$ is
bi-invariant if and only if the monoid $\G$ is cancellative: $ac = bc$ or
$ca = cb$ implies $a=b$. Now Malcev~\cite{Malcev} has constructed a
monoid with eight generators that is cancellative, but cannot map
injectively and homomorphically into any group (metric or not).

\begin{remark}[The idea of Malcev]
The preceding question occurred to us as a natural parallel to the
question answered negatively in this section:
\textit{Is every left-invariant metric semigroup strongly left-invariant,
i.e.\ does every left-invariant metric semigroup embed isometrically and
homomorphically into a left-invariant metric monoid?}
Moreover, we learned of Malcev's (negative) example shortly after having
shown the results in this paper. That said, not only are the two
questions similar, and their answers both negative, but interestingly,
even the approaches have a common philosophy. Namely, our approach in
Theorem~\ref{Tembed} and Proposition~\ref{Pcounterex} was to first show
that the only idempotent in a strongly left-invariant semigroup is a
(=the) two-sided identity; and to then construct a left-invariant
semigroup containing an idempotent that is not a two-sided identity. This
resembles (in spirit) Malcev's beautiful and more intricate idea
in~\cite{Malcev}, which was to first note that if a group contains eight
elements $a,b,c,d,u,v,x,y$ that satisfy the relations
\begin{equation}\label{Emalcev}
au=bv, \qquad ax=by, \qquad cx=dy,
\end{equation}
then one can successively solve to get
$vu^{-1} = b^{-1}a = yx^{-1} = d^{-1}c$,
and hence $cu = dv$. In the second step, Malcev constructed a monoid $M$
that is cancellative (note, this is if and only if the discrete metric on
$M$ is bi-invariant), and is finitely generated with generators
$a,b,c,d,u,v,x,y$ that satisfy~\eqref{Emalcev}, but in which $cu \neq
dv$. Hence $M$ cannot embed as a sub-monoid in any group.
\end{remark}

\subsection*{Acknowledgements}
I thank Terence Tao for the reference~\cite{Malcev}, and Bhaswar
Bhattacharya, Manjunath Krishnapur, Muna Naik, and Soumik Pal for useful
discussions. This work was partially supported by Ramanujan Fellowship
grant SB/S2/RJN-121/2017 and SwarnaJayanti Fellowship grants
SB/SJF/2019-20/14 and DST/SJF/MS/2019/3 from SERB and DST (Govt.~of
India).



\end{document}